\documentclass[12pt]{article}
\usepackage{amsmath,amsfonts,amssymb,amsthm}
\usepackage{enumerate}
\usepackage{caption}
\usepackage{pbox}

\numberwithin{equation}{section}
\theoremstyle{plain}
\newtheorem{theorem}[equation]{Theorem}
\newtheorem{corollary}[equation]{Corollary}
\newtheorem{lemma}[equation]{Lemma}
\newtheorem{proposition}[equation]{Proposition}

\theoremstyle{definition}
\newtheorem{definition}[equation]{Definition}

\newtheorem{remark}[equation]{Remark}

\numberwithin{equation}{section}

\newcommand{\R}{{\mathbb R}}
\newcommand{\N}{{\mathbb N}}

\newcommand{\Q}{{\mathbb Q}}

\newcommand{\Om}{\Omega}

\providecommand{\vint}[1]{\mathchoice
          {\mathop{\vrule width 5pt height 3 pt depth -2.5pt
                  \kern -9pt \kern 1pt\intop}\nolimits_{\kern -5pt{#1}}}
          {\mathop{\vrule width 5pt height 3 pt depth -2.6pt
                  \kern -6pt \intop}\nolimits_{\kern -3pt{#1}}}
          {\mathop{\vrule width 5pt height 3 pt depth -2.6pt
                  \kern -6pt \intop}\nolimits_{\kern -3pt{#1}}}
          {\mathop{\vrule width 5pt height 3 pt depth -2.6pt
                  \kern -6pt \intop}\nolimits_{\kern -3pt{#1}}}}

\newcommand{\eps}{\varepsilon}
\newcommand{\loc}{\mathrm{loc}}

\newcommand{\BV}{\mathrm{BV}}
\newcommand{\liploc}{\mathrm{Lip}_{\mathrm{loc}}}

\newcommand{\ch}{\text{\raise 1.3pt \hbox{$\chi$}\kern-0.2pt}}

\DeclareMathOperator{\capa}{Cap}
\DeclareMathOperator{\rcapa}{cap}

\DeclareMathOperator{\dist}{dist}
\DeclareMathOperator{\diam}{diam}
\DeclareMathOperator{\rad}{rad}
\DeclareMathOperator{\Lip}{Lip}

\DeclareMathOperator{\supp}{spt}

\DeclareMathOperator{\fint}{fine-int}

\begin{document}
\title{The Choquet and Kellogg properties for the fine topology when $p=1$ in metric spaces
\footnote{
\hfill \break {\it Keywords\,}: metric measure space, function of bounded variation,
$1$-fine topology,
fine Kellogg property, Choquet property, quasi-Lindel\"of principle
}}
\author{Panu Lahti}
\maketitle

\begin{abstract}
In the setting of a complete metric space that is equipped with a doubling
measure and supports a Poincar\'e
inequality, we prove the fine Kellogg property, the quasi-Lindel\"of principle, and the Choquet property for the fine topology in the case $p=1$.
\end{abstract}

\section{Introduction}

Much of nonlinear potential theory, for $1<p<\infty$, deals with $p$-harmonic functions, which are local minimizers of the $L^p$-norm of $|\nabla u|$.
Such minimizers can be defined also in metric measure spaces by using
\emph{upper gradients},
and the notion can be extended to the case $p=1$ by considering
\emph{functions of least gradient}, which are functions of bounded variation
($\BV$ functions) that minimize the total variation locally  ---
see Section \ref{preliminaries} for definitions.
Nonlinear potential theory for $1<p<\infty$ has by now reached a mature state even
in the general setting of metric spaces
that are equipped with a doubling measure
and support a Poincar\'e inequality, see especially the monograph
\cite{BB} and e.g. \cite{BB-OD,BBS2,BBS3,KiMa,S2}.
Functions of least gradient have been studied less, see however \cite{BDG,MRL,MST,SWZ}
for some previous works in the Euclidean setting, and \cite{HKLS,KLLS} in the metric setting.

In the case $1<p<\infty$, it is known that the \emph{p-fine topology} is the coarsest topology
that makes $p$-superharmonic functions continuous.
For nonlinear fine potential theory and its history in the Euclidean setting, for $1<p<\infty$,
see especially the monographs  \cite{AH,HKM,MZ}.
In the metric setting, fine potential theory for $1<p<\infty$
has been studied recently
in \cite{BBL-SS,BBL-CCK,BBL-WC}.
In these papers, the so-called Cartan and Choquet properties for the $p$-fine topology were proved,
and the latter was then used to deduce two further facts: first that every $p$-finely
open set is $p$-\emph{quasiopen}, and second that an arbitrary set
is $p$-\emph{thick} at
$p$-quasi-every of its points.
The latter fact is called the fine Kellogg property.

Few results of fine potential theory seem to have been considered in the case $p=1$, see however \cite{Tur} for some results in weighted Euclidean spaces.
The case $p=1$ is quite different since
cornerstones of the theory for $p>1$, such as comparison principles and
continuity of $p$-harmonic functions, are not available.
However, the author has previously studied some aspects of
fine potential theory for $p=1$ in metric spaces in \cite{L-Fed,L-FC,L-WC}.
The setting in these papers as well as the current one is a complete
metric space that is equipped with a doubling measure and supports a Poincar\'e inequality.
In \cite{L-WC}, the author proved a weak Cartan property in the case $p=1$,
see Theorem \ref{thm:weak Cartan property in text} below.
In this paper we prove, in the case $p=1$,
the fine Kellogg property and
the Choquet property in forms that are exactly analogous with the case $p>1$, see Corollary \ref{cor:fine Kellogg} and Theorem \ref{thm:Choquet property}.
Moreover, in Theorem \ref{thm:quasi-Lindelof} we prove the so-called
quasi-Lindel\"of principle for $1$-finely open sets,
also in a form that is exactly analogous with the case $p>1$.

In the case $1<p<\infty$, the different properties are deduced as follows,
see \cite{BBL-CCK,HKM-CFT,MZ}.
\begin{align*}
&\textrm{Cartan property}\implies\textrm{Choquet property}\\
&\qquad\qquad\implies
\begin{cases}
& \textrm{Fine Kellogg property}\implies\textrm{Quasi-Lindel\"of principle}\\
& \textrm{Finely open sets are quasiopen}
\end{cases}
\end{align*}

For $p>1$,
the fine Kellogg property is closely related to the (usual) Kellogg property,
which states that
$p$-quasi every boundary point of an open set is \emph{regular}, meaning
that $p$-harmonic solutions of the Dirichlet problem are continuous up to these boundary
points when the boundary data is continuous.
According to the Wiener criterion, every point where the complement
of an open set is
$p$-thick is a regular boundary point, which combined with the fine Kellogg property implies the (usual) Kellogg property.

In the case $p=1$,
since we lack various tools such as comparison principles but on the 
other hand have access to other more geometric tools, we deduce the various properties in a rather different order, as follows.
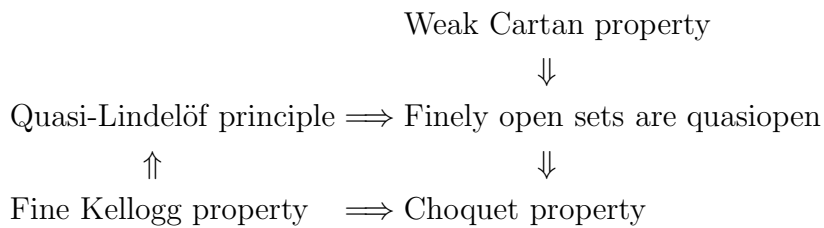
\begin{alignat*}{5}
&  &&  \textrm{Weak Cartan property}\\
&  &&  \qquad\qquad\Downarrow\\
&\textrm{Quasi-Lindel\"of principle} &\implies &\textrm{Finely open sets are quasiopen}\\
&\qquad\qquad \Uparrow &&  \qquad\qquad\Downarrow\\
&\textrm{Fine Kellogg property}  &\implies & \textrm{Choquet property}
\end{alignat*}

Easy examples such as that of a square show that in the case $p=1$
there can be many irregular boundary points and a Kellogg property does not hold.
We retain the term ``fine Kellogg property'' for $p=1$,
but this property is a starting point for proving the other properties rather than an end in itself, in contrast to the case $p>1$.
In terms of applications, one of our main motivations
is that the quasi-Lindel\"of principle will be useful in further research when considering
\emph{1-strict subsets} and partition of unity arguments in $1$-finely open sets.

\section{Preliminaries}\label{preliminaries}

In this section we introduce the notation, definitions,
and assumptions employed in the paper.

Throughout this paper, $(X,d,\mu)$ is a complete metric space that is equip\-ped
with a metric $d$ and a Borel regular outer measure $\mu$ satisfying
a doubling property, meaning that
there exists a constant $C_d\ge 1$ such that
\[
0<\mu(B(x,2r))\le C_d\mu(B(x,r))<\infty
\]
for every ball $B(x,r):=\{y\in X:\,d(y,x)<r\}$.
We also assume that $X$ consists of at least $2$ points.
Sometimes we abbreviate $B=B(x,r)$ and $aB:=B(x,ar)$ for $a>0$.
Note that in metric spaces, a ball (as a set) does not necessarily have a unique center and radius,
but we will always understand these to be  prescribed for the balls that we consider.

A complete metric space equipped with a doubling measure is proper,
that is, closed and bounded sets are compact.
Since $X$ is proper, for any open set $\Omega\subset X$
we define $\liploc(\Omega)$ to be the space of
functions that are Lipschitz in every open $\Omega'\Subset\Omega$.
Here $\Omega'\Subset\Omega$ means that $\overline{\Omega'}$ is a
compact subset of $\Omega$. Other local spaces of functions are defined analogously.

For any set $A\subset X$ and $0<R<\infty$, the restricted spherical Hausdorff content
of codimension one is defined to be
\[
\mathcal{H}_{R}(A):=\inf\left\{ \sum_{i=1}^{\infty}
\frac{\mu(B(x_{i},r_{i}))}{r_{i}}:\,A\subset\bigcup_{i=1}^{\infty}B(x_{i},r_{i}),\,r_{i}\le R\right\}.
\]
(We interpret $B(x,0)=\emptyset$ and $\mu(B(x,0))/0=0$, so that finite coverings
are also allowed.)
The codimension one Hausdorff measure of $A\subset X$ is then defined to be
\[
\mathcal{H}(A):=\lim_{R\rightarrow 0}\mathcal{H}_{R}(A).
\]

All functions defined on $X$ or its subsets will take values in $[-\infty,\infty]$.
By a curve we mean a nonconstant rectifiable continuous mapping from a compact
interval of the real line into $X$.
A nonnegative Borel function $g$ on $X$ is an upper gradient 
of a function $u$
on $X$ if for all nonconstant curves $\gamma$, we have
\[
|u(x)-u(y)|\le \int_\gamma g\,ds,
\]
where $x$ and $y$ are the end points of $\gamma$
and the curve integral is defined by using an arc-length parametrization,
see \cite[Section 2]{HK} where upper gradients were originally introduced.
We interpret $|u(x)-u(y)|=\infty$ whenever  
at least one of $|u(x)|$, $|u(y)|$ is infinite.
By only considering curves $\gamma$ in a set $A\subset X$,
we can talk about a function $g$ being an upper gradient of $u$ in $A$.

Given an open set $\Om\subset X$, we let
\[
\Vert u\Vert_{N^{1,1}(\Om)}:=\Vert u\Vert_{L^1(\Om)}+\inf \Vert g\Vert_{L^1(\Om)},
\]
where the infimum is taken over all $1$-weak upper gradients $g$ of $u$ in $\Om$.
The substitute for the Sobolev space $W^{1,1}$ in the metric setting is the Newton-Sobolev space
\[
N^{1,1}(\Om):=\{u:\|u\|_{N^{1,1}(\Om)}<\infty\},
\]
which was first considered in \cite{S}.
We understand every Newton-Sobolev function to be defined at every $x\in \Om$
(even though $\Vert \cdot\Vert_{N^{1,1}(\Om)}$ is then only a seminorm).

We will assume throughout the paper that $X$ supports a $(1,1)$-Poincar\'e inequality,
meaning that there exist constants $C_P>0$ and $\lambda \ge 1$ such that for every
ball $B(x,r)$, every $u\in L^1_{\loc}(X)$,
and every upper gradient $g$ of $u$,
we have
\[
\vint{B(x,r)}|u-u_{B(x,r)}|\, d\mu 
\le C_P r\vint{B(x,\lambda r)}g\,d\mu,
\]
where 
\[
u_{B(x,r)}:=\vint{B(x,r)}u\,d\mu :=\frac 1{\mu(B(x,r))}\int_{B(x,r)}u\,d\mu.
\]

The $1$-capacity of a set $A\subset X$ is given by
\[
\capa_1(A):=\inf \Vert u\Vert_{N^{1,1}(X)},
\]
where the infimum is taken over all functions $u\in N^{1,1}(X)$ such that
$u\ge 1$ on $A$.
We know that $\capa_1$ is an outer capacity, meaning that
\[
\capa_1(A)=\inf\{\capa_1(W):\,W\supset A\textrm{ is open}\}
\]
for any $A\subset X$, see e.g. \cite[Theorem 5.31]{BB}.
If a property holds outside a set
$A\subset X$ with $\capa_1(A)=0$, we say that it holds
at $1$-quasi-every point.

We say that a set $U\subset X$ is $1$-quasiopen if for every $\eps>0$ there is an
open set $G\subset X$ such that $\capa_1(G)<\eps$ and $U\cup G$ is open.

The variational $1$-capacity of a set $A\subset D$
with respect to a set $D\subset X$ is given by
\[
\rcapa_1(A,D):=\inf \int_X g_u \,d\mu,
\]
where the infimum is taken over functions $u\in N^{1,1}(X)$ such that $u\ge 1$ on $A$
and $u=0$ on $X\setminus D$, and their upper gradients $g_u$.
For basic properties satisfied by capacities, such as monotonicity and countable
subadditivity, see e.g. \cite{BB,BB-cap}.

Next we recall the definition and basic properties of functions
of bounded variation on metric spaces, following \cite{M}. See also e.g. \cite{AFP, EvaG92, Fed, Giu84, Zie89} for the classical 
theory in the Euclidean setting.
Let $\Om\subset X$ be an open set.
Given a function $u\in L^1_{\loc}(\Om)$, we define the total variation of $u$ in $\Om$ by
\[
\|Du\|(\Om):=\inf\left\{\liminf_{i\to\infty}\int_\Om g_{u_i}\,d\mu:\, u_i\in \Lip_{\loc}(\Om),\, u_i\to u\textrm{ in } L^1_{\loc}(\Om)\right\},
\]
where each $g_{u_i}$ is an upper gradient of $u_i$ in $\Om$.
(Note that in \cite{M}, local Lipschitz constants were used instead of upper
gradients, but the theory can be developed with either definition.)
If $u\in L^1(\Om)$ and $\Vert Du\Vert(\Om)<\infty$,
we say that $u$ is a function of bounded variation 
and denote $u\in\BV(\Om)$.
For an arbitrary set $A\subset X$, we define
\[
\|Du\|(A):=\inf\{\|Du\|(W):\, A\subset W,\,W\subset X
\text{ is open}\}.
\]
If $u\in L^1_{\loc}(\Om)$ and $\Vert Du\Vert(\Omega)<\infty$, $\|Du\|(\cdot)$ is
a Radon measure on $\Omega$ by \cite[Theorem 3.4]{M}.
A $\mu$-measurable set $E\subset X$ is said to be of finite perimeter if $\|D\ch_E\|(X)<\infty$, where $\ch_E$ is the characteristic function of $E$.
The perimeter of $E$ in $\Omega$ is also denoted by
\[
P(E,\Omega):=\|D\ch_E\|(\Omega).
\]

The lower and upper approximate limits of a function $u$ on $X$ are defined respectively by
\begin{equation}\label{eq:lower approximate limit}
u^{\wedge}(x):
=\sup\left\{t\in\R:\,\lim_{r\to 0}\frac{\mu(B(x,r)\cap\{u<t\})}{\mu(B(x,r))}=0\right\}
\end{equation}
and
\begin{equation}\label{eq:upper approximate limit}
u^{\vee}(x):
=\inf\left\{t\in\R:\,\lim_{r\to 0}\frac{\mu(B(x,r)\cap\{u>t\})}{\mu(B(x,r))}=0\right\}.
\end{equation}
Unlike Newton-Sobolev functions, we understand $\BV$ functions to be
$\mu$-equivalence classes. To study fine properties, we need to
consider the pointwise representatives $u^{\wedge}$ and $u^{\vee}$.
By Lebesgue's differentiation theorem (see e.g. \cite[Chapter 1]{Hei}),
for any $u\in L^1_{\loc}(X)$ we have $u=u^{\wedge}=u^{\vee}$ $\mu$-almost everywhere.

The $\BV$-capacity of a set $A\subset X$ is defined by
\begin{equation}\label{eq:BVcapacity}
\capa_{\BV}(A):=\inf \Vert u\Vert_{\BV(X)},
\end{equation}
where the infimum is taken over all $u\in\BV(X)$ with $u\ge 1$ in a neighborhood of $A$.
The $\BV$-capacity has the following useful continuity property
(not satisfied by the $1$-capacity):
by \cite[Theorem 3.4]{HaKi} we know that if
$A_1\subset A_2\subset \ldots\subset X$, then
\begin{equation}\label{eq:continuity of BVcap}
\capa_{\BV}\left(\bigcup_{i=1}^{\infty}A_i\right)=\lim_{i\to\infty} \capa_{\BV}(A_i).
\end{equation}
On the other hand, by \cite[Theorem 4.3]{HaKi} we know that for some constant
$C(C_d,C_P,\lambda)\ge 1$ and any
$A\subset X$, we have
\begin{equation}\label{eq:Newtonian and BV capacities are comparable}
\capa_{\BV}(A)\le \capa_1(A)\le C\capa_{\BV}(A);
\end{equation}
whenever we want to state that a constant $C$
depends on the parameters $a,b, \ldots$, we write $C=C(a,b,\ldots)$.

By combining this with
\cite[Theorem 5.1]{HaKi}, we get for any $A\subset X$ that
\begin{equation}\label{eq:null sets of Hausdorff measure and capacity}
\capa_{\BV}(A)=0\quad\textrm{iff}\quad\capa_1(A)=0\quad\textrm{iff}\quad\mathcal H(A)=0.
\end{equation}

Next we define the fine topology in the case $p=1$.
\begin{definition}\label{def:1 fine topology}
We say that $A\subset X$ is $1$-thin at the point $x\in X$ if
\[
\lim_{r\to 0}r\frac{\rcapa_1(A\cap B(x,r),B(x,2r))}{\mu(B(x,r))}=0.
\]
We say that a set $U\subset X$ is $1$-finely open if $X\setminus U$ is $1$-thin at every $x\in U$. Then we define the $1$-fine topology as the collection of $1$-finely open sets on $X$.

We denote the $1$-fine interior of a set $D\subset X$, i.e. the largest $1$-finely open set contained in $D$, by $\fint D$. We denote the $1$-fine closure of $D$, i.e. the smallest $1$-finely closed set containing $D$, by $\overline{D}^1$.

Finally, we define the \emph{1-base} $b_1 D$ of a set $D\subset X$ as the set of points where $D$ is 1-\emph{thick}, i.e. not $1$-thin.
\end{definition}

See \cite[Section 4]{L-FC} for a proof of the fact that the
$1$-fine topology is indeed a topology.
A rather differently formulated notion of $1$-thinness was previously given
in the weighted Euclidean
setting in \cite{Tur}.

The support of a ($\mu$-almost everywhere defined) function $u$ on $X$ is the closed set
\[
\supp u:=\{x\in X:\,\mu(\{u\neq 0\}\cap B(x,r))>0\ \textrm{for all }r>0\}.
\]
For an open set $\Om\subset X$, we denote by $\BV_c(\Om)$ the class of functions $\varphi\in\BV(\Om)$ with
compact support in $\Om$, that is, $\supp \varphi\Subset \Om$.

\begin{definition}\label{def:least gradient}
Let $\Om\subset X$ be an open set.
We say that $u\in\BV_{\loc}(\Om)$ is a $1$-minimizer  in $\Om$ if
for all $\varphi\in \BV_c(\Om)$,
\begin{equation}\label{eq:definition of 1minimizer}
\Vert Du\Vert(\supp\varphi)\le \Vert D(u+\varphi)\Vert(\supp\varphi).
\end{equation}
We say that $u\in\BV_{\loc}(\Om)$ is a $1$-superminimizer in $\Om$
if \eqref{eq:definition of 1minimizer} holds for all nonnegative $\varphi\in \BV_c(\Om)$.
\end{definition}

In the literature, $1$-minimizers are usually called functions of least gradient. 

\section{Subsets of finite Hausdorff measure}

It is a well known problem to find a subset of strictly positive but finite Hausdorff measure
from a set of infinite Hausdorff measure, see \cite{How,Lar,Rog}.
We will need such a property when proving the fine Kellogg property,
but unfortunately the existing results do not seem to directly apply
to the codimension $1$ Hausdorff measure $\mathcal H$.
The reason is that the quantity $\mu(B(x,r))/r$ is not necessarily
increasing with respect to $r$, which is usually taken as a standard assumption.
Thus in this section we study the existence of subsets of finite Hausdorff measure $\mathcal H$.
In doing this, we are nonetheless able to follow almost directly the argument
presented in \cite[Chapter 8]{PM}, which is based on \cite{How}.

For any Radon measure $\nu$ on $X$ (we understand measures to be positive)
and $R>0$, define the maximal function
\[
\mathcal M_R \nu(x):=\sup_{0<r\le R}r\frac{\nu(B(x,r))}{\mu(B(x,r))},\quad x\in X.
\]

First we need a version of Frostman's lemma, which fortunately has been proved also for
codimension Hausdorff measures ---
the following is a special case of \cite[Theorem 6.1]{Mal}.
Note that we can always understand Radon measures on open (or more generally
Borel) sets to be defined on the
whole space $X$.

\begin{theorem}\label{thm:Frostman for open sets}
Let $W\subset X$ be an open set and $R>0$. Then there exists a Radon measure $\nu$ on $W$
such that $\mathcal M_R \nu\le 1$ on $X$ and $\mathcal H_{10R}(W)\le C_F \nu(W)$
for some constant $C_F=C_F(C_d)$.
\end{theorem}

We easily get the following version for compact sets.

\begin{theorem}\label{thm:Frostman for compact sets}
	Let $K\subset X$ be compact and let $R>0$. Then there exists a Radon measure $\nu$ on $K$
	such that $\mathcal M_R \nu\le 1$ on $X$ and $\mathcal H_{10R}(K)\le C_F \nu(K)$.
\end{theorem}

\begin{proof}
For any $a>0$, let $K^a:=\{x\in X:\,\dist(x,K)<a\}$.
For each open set $W_i:=K^{1/i}$, $i\in\N$, apply Theorem \ref{thm:Frostman for open sets}
to obtain a Radon measure $\nu_i$ on $W_i$ such that $\mathcal M_R \nu_i\le 1$ on
$X$ and $\mathcal H_{10R}(W_i)\le C_F \nu_i(W_i)$.
From the condition $\mathcal M_R \nu_i\le 1$ and the compactness of $K$ it easily follows that
$\nu_i(X)$ is a bounded sequence, and so
there exists a subsequence (not relabeled)
and a Radon measure $\nu$ on $X$ such that $\nu_i\overset{*}{\rightharpoonup}\nu$ on $X$
(see e.g. \cite[Theorem 1.59]{AFP}).
By lower semicontinuity in open sets under weak* convergence
(see e.g. \cite[Proposition 1.62]{AFP}), for every $j\in\N$ we have
\[
\nu(X\setminus \overline{W_j})\le \liminf_{i\to\infty}
\nu_i(X\setminus \overline{W_j})=0,
\]
and it follows that $\nu(X\setminus K)=0$, that is, $\nu$ is a Radon measure
on $K$. Moreover, for any $x\in X$ and $0<r\le R$,
\[
r\frac{\nu(B(x,r))}{\mu(B(x,r))}
\le \liminf_{i\to\infty}r\frac{\nu_i(B(x,r))}{\mu(B(x,r))}\le 1,
\]
and so $\mathcal M_R \nu\le 1$ on $X$.
Finally, by upper semicontinuity in compact sets,
\begin{align*}
\nu(K)=\nu(\overline{W_1})\ge\limsup_{i\to\infty}\nu_i(\overline{W_1})
&=\limsup_{i\to\infty}\nu_i(W_i)\\
&\ge \limsup_{i\to\infty}\frac{\mathcal H_{10R}(W_i)}{C_F}
\ge\frac{\mathcal H_{10R}(K)}{C_F}.
\end{align*}
\end{proof}

Now we prove the existence of subsets of positive finite $\mathcal H$-measure.
We make no attempt to give the most general possible result (one that might cover e.g. Hausdorff measures
of a different codimension), but rather just prove a version that
will suffice for our purposes.

\begin{theorem}\label{thm:existence of subset}
Let $K\subset X$ be a compact set with $\mathcal H(\{x\})=0$ for all $x\in K$. Then
\[
\mathcal H(K)=\sup \{\mathcal H(K_0):\, K_0\subset K\textrm{ is compact with }\mathcal H(K_0)<\infty\}.
\]
\end{theorem}

\begin{proof}
If $\mathcal H(K)<\infty$, the result is obvious.
Thus we can assume that $\mathcal H(K)=\infty$.
By the compactness of $K$, for each $k\in \N$ we find a finite family
of balls $\{B_{k,j}:=B(x_{k,j},2^{-k})\}_{j=1}^{m_k}$, $m_k\in\N$, such
that the balls $\tfrac 12 B_{k,j}=B(x_{k,j},2^{-k-1})$ cover $K$.
Let
\[
\mathcal B:=\{B_{k,j}:\, k\in\N,\,j\in\{1,\ldots,m_k\}\}.
\]
Fix $0<M<\infty$. Since $\mathcal H(K)=\infty$, we find and fix $0<\delta<1/2$ such that
$\mathcal H_{10\delta}(K)>M$. By Theorem \ref{thm:Frostman for compact sets}
we find a Radon measure $\nu_0$ on $K$ such that
$\mathcal M_{\delta}\nu_0\le 1$ on $X$ and $\mathcal H_{10 \delta}(K)\le C_F\nu_0(K)$.
For any ball $B=B(x,r)$, denote $\rad(B):=r$.
Let $\mathcal F_{\delta}$ be the set of all Radon measures
$\nu$ on $K$ for which
\[
\rad(B)\frac{\nu(B)}{\mu(B)}\le 1\quad\textrm{for all }B\in \mathcal B\textrm{ with}\,\rad(B)\le\delta.
\]
Let
\[
h:=\sup\{\nu(K):\, \nu\in \mathcal F_{\delta}\}
\]
and
\[
\mathcal G_{\delta}:=\{\nu\in \mathcal F_{\delta}:\, \nu(K)=h\}.
\]
Then $h> M/C_F$, since $\nu_0\in \mathcal F_{\delta}$.
Pick a sequence $(\nu_i)\subset \mathcal F_{\delta}$ such that
$\nu_i(K)\to h$. As $\nu_i(X)$ is clearly bounded
(by the definition of $\mathcal F_{\delta}$ and the compactness of $K$),
we find a subsequence (not relabeled) such
that $\nu_i\overset{*}{\rightharpoonup}\widetilde{\nu}$ for some Radon measure
$\widetilde{\nu}$ on $X$.
By lower semicontinuity in open sets under weak* convergence,
\[
\widetilde{\nu}(X\setminus K)\le \liminf_{i\to\infty}\nu_i(X\setminus K)=0,
\]
and for any $B\in\mathcal B$ with $\rad(B)\le\delta$,
\[
\rad(B)\frac{\widetilde{\nu}(B)}{\mu(B)}
\le \liminf_{i\to\infty}\rad(B)\frac{\nu_i(B)}{\mu(B)}\le 1,
\]
and so
$\widetilde{\nu}\in \mathcal F_{\delta}$. By upper semicontinuity
in compact sets,
\[
\widetilde{\nu}(K)\ge \limsup_{i\to\infty}\nu_i(K)=h,
\]
and so necessarily $\widetilde{\nu}(K)=h$.
Thus $\mathcal G_{\delta}$ is nonempty, it is easily seen to be convex,
and by using the properties of weak* convergence as above it can be verified that
$\mathcal G_{\delta}$ is compact with respect to the weak* topology.
Then by the Krein-Milman theorem,
see e.g. \cite[Theorem 3.23]{Rud}, there exists an extreme point
$\nu\in \mathcal G_{\delta}$;
this means that if $\nu=t\nu_1+(1-t)\nu_2$
for some $0<t<1$ and
$\nu_1,\nu_2\in \mathcal G_{\delta}$, then $\nu=\nu_1=\nu_2$.

Fix $0<\eps<\delta$. Define
\[
D_{\eps}:=\bigcup\left\{B\in \mathcal B:\,\rad(B)\le
\eps\textrm{ and } 2 \nu(B)>\mu(B)/\rad(B)  \right\}.
\]
We claim that $\nu(K\setminus D_{\eps})=0$.

Suppose that $\nu(K\setminus D_{\eps})>0$ instead. Let $B_1,\ldots,B_m$ be all the balls
in $\mathcal B$ with radius is strictly greater than $\eps$.
Define inductively for $i=1,\ldots,m$,
\begin{align*}
& A_1:=K\setminus D_{\eps},\\
& A_{i+1}:=A_i\setminus B_i\quad \textrm{if }\nu(A_i\setminus B_i)\ge \nu(A_i\cap B_i),\\
& A_{i+1}:=A_i\cap B_i\quad \textrm{if }\nu(A_i\setminus B_i)< \nu(A_i\cap B_i).
\end{align*}
Let $A:=A_{m+1}$. Then $A\subset K$ is a Borel set and $\nu(A)>0$. Moreover,
either $A\subset B$ or $A\cap B=\emptyset$ for every $B\in \mathcal B$ with $\rad (B)>\eps$.
Recall that $\mathcal H(\{x\})=0$ for all $x\in K$; then by the doubling property
of $\mu$, necessarily
\begin{equation}\label{eq:Hausdorff measure zero consequence}
\liminf_{r\to 0}\frac{\mu(B(x,r))}{r}=0.
\end{equation}
Now if $x\in K$, for all balls $B\in\mathcal B$,
$B\ni x$ with $\rad(B)\le \delta$ we have
\[
\nu(\{x\})\le \nu(B)\le \frac{\mu(B)}{\rad(B)},
\]
where the last quantity can be made arbitrarily small by
\eqref{eq:Hausdorff measure zero consequence}.
We conclude that $\nu(\{x\})=0$ for all $x\in K$.
Thus we can find disjoint Borel sets $H_1$ and $H_2$ such that
\[
A=H_1\cup H_2\quad\textrm{and}\quad\nu(H_1)=\nu(H_2)=\tfrac 12 \nu(A),
\]
see e.g. \cite[Lemma 8.20]{PM}.
Now we define Radon measures $\nu_1$ and $\nu_2$ by
\begin{align*}
& \nu_1(E):= 2\nu(E\cap H_1)+\nu(E\setminus A),\\
& \nu_2(E):= 2\nu(E\cap H_2)+\nu(E\setminus A)
\end{align*}
for $E\subset K$. Then $\nu=(\nu_1+\nu_2)/2$,
$\nu_1(H_1)=2 \nu(H_1)>\nu(H_1)$, and $\nu_2(H_2)=2 \nu(H_2)>\nu(H_2)$.
Thus by the extremality of $\nu$, $\nu_1$ and $\nu_2$ cannot both belong to $\mathcal G_{\delta}$.
Suppose $\nu_1\notin \mathcal G_{\delta}$, the other case being treated analogously.
Clearly $\nu_1(K)=\nu_2(K)=\nu(K)=h$.
Thus $\nu_1\notin \mathcal F_{\delta}$, that is, there exists $B\in \mathcal B$ with $\rad(B)\le \delta$ such that
$\nu_1(B)>\mu(B)/\rad(B)$.
Then $\rad(B)\le \eps$ since otherwise either $A\subset B$ or $A\cap B=\emptyset$,
and in both cases $\nu_1(B)=\nu(B)\le \mu(B)/\rad(B)$. We have
\[
2 \nu(B)
\ge 2 \nu(B\cap H_1)+\nu(B\setminus A)
=\nu_1(B)
>\mu(B)/\rad(B),
\]
whence $B\subset D_{\eps}\subset X\setminus A$. Thus in total,
\[
\frac{\mu(B)}{\rad(B)}<\nu_1(B)=\nu(B)\le \frac{\mu(B)}{\rad(B)},
\]
a contradiction. Thus $\nu(K\setminus D_{\eps})=0$.

Let $p\in \N$ such that $1/p<\delta$ and
\[
D:=\bigcap_{i=p}^{\infty} D_{1/i},
\]
so that $\nu(K\setminus D)=0$ by the claim.
Fix a new $\eps>0$. For every $x\in D$ there exists $B\in \mathcal B$ such that $x\in B$,
$\rad(B)\le \eps/5$, and
$2\nu(B)> \mu(B)/\rad(B)$.
By the $5$-covering theorem, we can extract a countable collection $\{B_i\}_{i=1}^{\infty}$
of pairwise disjoint balls such that the balls $5B_i$ cover $D$.
Then
\begin{align*}
\mathcal H_{\eps}(D) \le \sum_{i=1}^{\infty} \frac{\mu(5B_i)}{5\rad (B_i)}
\le C_d^3 \sum_{i=1}^{\infty} \frac{\mu(B_i)}{\rad (B_i)}
\le 2 C_d^3 \sum_{i=1}^{\infty} \nu(B_i)
\le 2 C_d^3 \nu(K),
\end{align*}
so that $\mathcal H(D)\le 2 C_d^3 \nu(K)<\infty$.
Conversely, for any balls $\{\widetilde{B}_i\}_{i=1}^{\infty}$ with
$\rad(\widetilde{B}_i)\le\delta/8$
covering $K\cap D$, such that $\widetilde{B}_i\cap K\neq\emptyset$
for all $i\in\N$, we find balls
$\{B_i\in\mathcal B\}_{i=1}^{\infty}$ for which $B_i\supset \widetilde{B}_i$
and $4\rad(\widetilde{B}_i)\le\rad(B_i)\le 8\rad(\widetilde{B}_i)$
for all $i\in\N$, which gives
\begin{align*}
C_d^4\sum_{i=1}^{\infty}\frac{\mu(\widetilde{B}_i)}{\rad(\widetilde{B}_i)}
&\ge \sum_{i=1}^{\infty}\frac{\mu(B_i)}{\rad(B_i)}\\
&\ge \sum_{i=1}^{\infty}\nu(B_i)
\ge \nu(K\cap D)= \nu(K)=h> M/C_F.
\end{align*}
Thus $M/(C_d^4 C_F)<\mathcal H(K\cap D)<\infty$. Then since
$\mathcal H|_{K\cap D}$ is a Radon measure,
we find a compact $K_0\subset K\cap D$ such that
$\mathcal H(K_0)> M/(C_d^4 C_F)$ (see e.g. \cite[Proposition 1.43(i)]{AFP}).
This completes the proof since $0<M<\infty$ was arbitrary.
\end{proof}

\section{The fine Kellogg property}

In this section we prove the fine Kellogg property for $p=1$; it states that
an arbitrary set is $1$-thick at $1$-quasi-every point in the set.
The proof will mostly consist of considerations of measurability, as well as
exploiting the existence of subsets of finite Hausdorff measure.

Recall the definition of the $\BV$-capacity from \eqref{eq:BVcapacity}.

\begin{lemma}\label{lem:Borel measurability}
Let $D\subset X$. Then the functions
\[
x\mapsto r\frac{\capa_{\BV}(D\cap B(x,r))}{\mu(B(x,r))}\quad\textrm{for a fixed }r>0,
\]
as well as
\[
x\mapsto \limsup_{r\to 0}r\frac{\capa_{\BV}(D\cap B(x,r))}{\mu(B(x,r))}\quad\textrm{and}\quad
x\mapsto \liminf_{r\to 0}\frac{\mu(B(x,r))}{r}
\]
are Borel measurable.
\end{lemma}
\begin{proof}
Fix $r>0$.
First we show that the function
\[
x\mapsto \capa_{\BV}(D\cap B(x,r))
\]
is lower semicontinuous.
By using e.g. Lipschitz cutoff functions, we see that the function is finite for all $x\in X$.
Fix $\eps>0$ and $x\in X$.
By the continuity of $\capa_{\BV}$ under increasing sequences of sets,
see \eqref{eq:continuity of BVcap}, we have for some $s<r$
\[
\capa_{\BV}(D\cap B(x,r))\le \capa_{\BV}(D\cap B(x,s))+\eps
\le \liminf_{y\to x}\capa_{\BV}(D\cap B(y,r))+\eps.
\]
Since $\eps>0$ was arbitrary, we have shown lower semicontinuity. Similarly, it can be shown that
\[
x\mapsto \mu(B(x,r))
\]
is lower semicontinuous, and so we have shown that for a fixed $r>0$, the function
\[
x\mapsto r\frac{\capa_{\BV}(D\cap B(x,r))}{\mu(B(x,r))}
\]
is Borel measurable. Moreover, again using the continuity of
$\capa_{\BV}$ under increasing sequences of sets
we see that for any $x\in X$ and $r>0$,
\[
\sup_{0<s< r}s\frac{\capa_{\BV}(D\cap B(x,s))}{\mu(B(x,s))}
=\sup_{\substack{0<s< r\\ s\in\Q}}s\frac{\capa_{\BV}(D\cap B(x,s))}{\mu(B(x,s))},
\]
and so
\[
x\mapsto \limsup_{r\to 0}r\frac{\capa_{\BV}(D\cap B(x,r))}{\mu(B(x,r))}
=\lim_{i\to\infty}\sup_{\substack{0<s< 1/i\\ s\in\Q}}s\frac{\capa_{\BV}(D\cap B(x,s))}{\mu(B(x,s))}
\]
is Borel measurable. The Borel measurability of the third function is shown similarly.
\end{proof}

\begin{lemma}\label{lem:positive Hausdorff measure points Borel}
The set $A:=\{x\in X:\,\mathcal H(\{x\})>0\}$ is Borel.
\end{lemma}
\begin{proof}
It is easy to check that $\mathcal H(\{x\})>0$ if and only if
\[
\liminf_{r\to 0}\frac{\mu(B(x,r))}{r}>0.
\]
Thus the result follows from Lemma \ref{lem:Borel measurability}.
\end{proof}

The $\BV$-capacity has the following ``Borel regularity''.

\begin{lemma}\label{lem:BV capacity Borel regularity}
	Let $A\subset X$. Then there exists a Borel set $D\supset A$ such that
	$\capa_{\BV}(D\cap B)=\capa_{\BV}(A\cap B)$ for every ball $B=B(x,r)$.
\end{lemma}
\begin{proof}
	Take a countable dense set $\{x_j\}_{j=1}^{\infty}$ in $X$ and
	let $\{B_k\}_{k=1}^{\infty}$ be the collection of all balls with center $x_j$
	for some $j$ and a rational radius.
	Define the set function
	\[
	\widetilde{\capa}_{\BV}(H):=\sum_{k=1}^{\infty}2^{-k}
	\frac{\capa_{\BV}(H\cap B_k)-\capa_{\BV}(A\cap B_k)}{\capa_{\BV}(B_k)},\quad H\supset A.
	\]
	Let $\beta:=\inf\{\widetilde{\capa}_{\BV}(H):\, H\supset A,\, H\ \textrm{is Borel}\}$.
	If we take a sequence of Borel sets $D_i\supset A$ such that
	$\widetilde{\capa}_{\BV}(D_i)<\beta+1/i$, then clearly $D:=\bigcap_{i=1}^{\infty}D_i$
	is a Borel set satisfying $\widetilde{\capa}_{\BV}(D)=\beta$.
	
	We show that $\beta=0$;
	suppose instead $\beta>0$. Then for some $k\in\N$ the $k$:th term in the definition
	of $\widetilde{\capa}_{\BV}(D)$ is nonzero.
	By the definition of the $\BV$-capacity, we find an open set $W\supset A\cap B_k$ with
	$\capa_{\BV}(W)<\capa_{\BV}(D\cap B_k)$.
	Then defining $D_0:=D\cap ((X\setminus B_k)\cup W)$, we have
	$\widetilde{\capa}_{\BV}(D_0)<\widetilde{\capa}_{\BV}(D)$,
	a contradiction.
	Thus $\beta=0$.
	
	Now take $x\in X$ and $r>0$. Let $\eps>0$.
	By \eqref{eq:continuity of BVcap}, for some $s<r$ we have
	\[
	\capa_{\BV}(D\cap B(x,r))\le \capa_{\BV}(D\cap B(x,s))+\eps,
	\]
	and then for some $k\in\N$ we have $B(x,s)\subset B_k\subset B(x,r)$.
	Then
	\begin{align*}
	\capa_{\BV}(D\cap B(x,r))
	&\le \capa_{\BV}(D\cap B(x,s))+\eps\\
	&\le \capa_{\BV}(D\cap B_k)+\eps\\
	&= \capa_{\BV}(A\cap B_k)+\eps\quad\textrm{since }\beta=0\\
	&\le \capa_{\BV}(A\cap B(x,r))+\eps.
	\end{align*}
	Since $\eps>0$ was arbitrary, we have the result.
\end{proof}

\begin{proposition}[{\cite[Proposition 4.5]{L-WC}}]\label{prop:positive capacity implies thickness}
	Let $x\in X$ with $\capa_1(\{x\})>0$. Then $\{x\}$ is $1$-thick at $x$, that is,
	$x\in b_1\{x\}$.
\end{proposition}

According to \cite[Proposition 6.16]{BB}, if $x\in X$, $0<r<\frac 18\diam X$,
and $A\subset B(x,r)$, then
for some constant $C=C(C_d,C_P,\lambda)$,
\begin{equation}\label{eq:comparison of capacities}
\frac{\capa_1(A)}{C(1+r)}\le \rcapa_1(A,B(x,2r))\le 2\left(1+\frac{1}{r}\right)\capa_1(A).
\end{equation}

Now we prove the existence of a thickness point in any set of nonzero
$1$-capacity, which will immediately imply the fine Kellogg property.
In the proof below, the reason for using the $\BV$-capacity is
that it is continuous
with respect to increasing sequences of sets, which in particular allowed us to prove the measurability results of Lemma \ref{lem:Borel measurability}
and Lemma \ref{lem:BV capacity Borel regularity}.

\begin{theorem}\label{thm:fine Kellogg property}
Let $A\subset X$ with $\capa_1(A)>0$.
Then there exists a point $x\in A\cap b_1 A$.
\end{theorem}

\begin{proof}
Suppose $A\cap b_1 A= \emptyset$, so that
\[
\lim_{r\to 0}r\frac{\rcapa_1(A\cap B(x,r),B(x,2r))}{\mu(B(x,r))}= 0
\]
for all $x\in A$.
By Proposition \ref{prop:positive capacity implies thickness} and
\eqref{eq:null sets of Hausdorff measure and capacity} we know that $\mathcal H(\{x\})=0$
for all $x\in A$.
By \eqref{eq:Newtonian and BV capacities are comparable},
$\capa_{\BV}(A)>0$.
By \eqref{eq:comparison of capacities} and \eqref{eq:Newtonian and BV capacities are comparable},
\[
\lim_{r\to 0}r\frac{\capa_{\BV}(A\cap B(x,r))}{\mu(B(x,r))}= 0
\]
for all $x\in A$.
By Lemma \ref{lem:BV capacity Borel regularity} we find
a Borel set $D\supset A$ such that
$\capa_{\BV}(D\cap B)=\capa_{\BV}(A\cap B)$ for every ball $B=B(x,r)$.
Thus for all $x\in A$,
\[
\lim_{r\to 0}r\frac{\capa_{\BV}(D\cap B(x,r))}{\mu(B(x,r))}=0.
\]
Then define
\[
F:=\left\{x\in X:\, \lim_{r\to 0}r\frac{\capa_{\BV}(D\cap B(x,r))}{\mu(B(x,r))}=0\right\}\supset A,
\]
which is a Borel set by Lemma \ref{lem:Borel measurability}.
Since we had $\mathcal H(\{x\})=0$ for all $x\in A$, we can define
\[
H:=D\cap F\cap \{x\in X:\,\mathcal H(\{x\})=0\}\supset A,
\]
which is a Borel set by Lemma \ref{lem:positive Hausdorff measure points Borel}.
For every $x\in H$ we now have
\[
\lim_{r\to 0}r\frac{\capa_{\BV}(H\cap B(x,r))}{\mu(B(x,r))}
\le \lim_{r\to 0}r\frac{\capa_{\BV}(D\cap B(x,r))}{\mu(B(x,r))}=0.
\]
Moreover, $\capa_{\BV}(H)>0$ since $A\subset H$.
(The point so far has simply been to pass from $A$ to the Borel set $H$.)

By the fact that $\capa_{\BV}$ is a Choquet capacity,
see \cite[Corollary 3.8]{HaKi}, we find a compact set $K\subset H$ with $\capa_{\BV}(K)>0$.
By \eqref{eq:null sets of Hausdorff measure and capacity}, $\mathcal H(K)>0$.
Now by Theorem \ref{thm:existence of subset}
we find a compact set $K_0\subset K$ with $0<\mathcal H(K_0)<\infty$
and then clearly
for every $x\in K_0$,
\[
\lim_{r\to 0}r\frac{\capa_{\BV}(K_0\cap B(x,r))}{\mu(B(x,r))}=0.
\]
By Egorov's theorem, which we can apply by Lemma \ref{lem:Borel measurability},
we find $K_1\subset K_0$ with $0<\mathcal H(K_1)<\infty$
such that
\[
r\frac{\capa_{\BV}(K_1\cap B(x,r))}{\mu(B(x,r))}\to 0\quad\textrm{as }r\to 0 
\]
uniformly for all $x\in K_1$. Fix $\eps>0$. For some $\delta>0$, the above quantity
is less than $\eps$ whenever
$r<\delta$. We find a covering $\{B(x_i,r_i)\}_{i=1}^{\infty}$ of $K_1$,
with $r_i<\delta/2$, such that
\[
\sum_{i=1}^{\infty}\frac{\mu(B(x_i,r_i))}{r_i}<\mathcal H(K_1)+\eps.
\]
We can assume that there exists $y_i\in B(x_i,r_i)\cap K_1$ for all $i\in\N$, and then the balls $B(y_i,2r_i)$ also cover $K_1$.
Thus
\begin{align*}
\capa_{\BV}(K_1)
&\le \sum_{i=1}^{\infty}\capa_{\BV}(K_1\cap B(y_i,2r_i))\\
&\le \eps\sum_{i=1}^{\infty}\frac{\mu(B(y_i,2r_i))}{2r_i}\\
&\le \eps\sum_{i=1}^{\infty}\frac{\mu(B(x_i,4r_i))}{2r_i}
<C_d^2\eps(\mathcal H(K_1)+\eps).
\end{align*}
Letting $\eps\to 0$,
we get $\capa_{\BV}(K_1)=0$, so that $\mathcal H(K_1)=0$ by \eqref{eq:null sets of Hausdorff measure and capacity}, which is a contradiction.
\end{proof}

Now we get the following fine Kellogg property for $p=1$.
For the case $p>1$, see \cite[Corollary 1.3]{BBL-CCK}.

\begin{corollary}\label{cor:fine Kellogg}
Let $A\subset X$. Then $\capa_1(A\setminus b_1 A)=0$.
\end{corollary}
\begin{proof}
Let $F:=A\setminus b_1 A$. Then clearly $b_1 F\subset b_1 A$, so that
$F\cap b_1 F\subset b_1 F\setminus b_1 A=\emptyset$, and then
by Theorem \ref{thm:fine Kellogg property}, $\capa_1(F)=0$.
\end{proof}

\begin{remark}
In this section we have not really studied or applied
fine potential theory, but rather just
basic properties of capacities and the measure-theoretic result of the previous section.
By contrast, in the case $p>1$, the fine Kellogg property is
deduced from the Choquet property, which we only prove for $p=1$ in Section \ref{sec:choquet}.
The kind of method we used in this section does not
seem to be available in
the case $p>1$: we used the fact that $\mathcal H$ and $\capa_{1}$
have the same null sets but the analog of this is not true for $p>1$.
\end{remark}

\section{The quasi-Lindel\"of principle}

In this section we prove the quasi-Lindel\"of principle
for the $1$-fine topology, by using the fine Kellogg property.

We will need the following fact given in \cite[Lemma 11.22]{BB}.
\begin{lemma}\label{lem:capacity wrt different balls}
	Let $x\in X$, $r>0$, and $A\subset B(x,r)$. Then for every $1<s<t$
	with $tr<\frac 14 \diam X$, we have
	\[
	\rcapa_1(A,B(x,tr))\le \rcapa_1(A,B(x,sr))\le C_S\left(1+\frac{t}{s-1}\right)\rcapa_1(A,B(x,tr)),
	\]
	where $C_S=C_S(C_d,C_P,\lambda)$.
\end{lemma}

In proving the quasi-Lindel\"of principle for $p=1$,
we follow the proof of \cite[Theorem 2.3]{HKM-CFT}, where the property was shown
for $p>1$ (in the Euclidean setting and with slightly different definitions).

\begin{theorem}\label{thm:quasi-Lindelof}
For every family $\mathcal V$ of $1$-finely open sets there is a countable subfamily
$\mathcal V'$ such that
\[
\capa_1\left(\bigcup_{U\in\mathcal V} U\setminus \bigcup_{U'\in\mathcal V'}U'\right)=0.
\]
\end{theorem}

\begin{proof}
Take a countable dense set $\{x_j\}_{j=1}^{\infty}$ in $X$ and
let $\{B_k\}_{k=1}^{\infty}$ be the collection of all balls with center $x_j$ for some $j$ and a
rational radius strictly less than $\tfrac{1}{16} \diam X$. Define the set function
\[
\widetilde{\capa}_1(A):=\sum_{k=1}^{\infty}2^{-k}\frac{\rcapa_1(A\cap B_k,2B_k)}{\rcapa_1(B_k,2B_k)},
\quad A\subset X,
\]
where the denominator is always strictly positive by \eqref{eq:comparison of capacities}.
Take a collection of $1$-finely open sets $\{U_i\}_{i\in \Lambda}$,
and let $U:=\bigcup_{i\in \Lambda}U_i$.
Let
\[
\beta:=\inf\left\{\widetilde{\capa}_1\left(U\setminus \bigcup_{i\in I}U_i\right):\,I\subset \Lambda\textrm{ is countable}\right\}.
\]
Choose countable sets $I_j\subset \Lambda$, $j\in\N$, such that
\[
\widetilde{\capa}_1\left(U\setminus \bigcup_{i\in I_j}U_i\right)<\beta+\frac 1j.
\]
Define the countable set
\[
I_{\infty}:=\bigcup_{j=1}^{\infty} I_j,
\]
so that for
\[
A:=U\setminus \bigcup_{i\in I_{\infty}} U_i
\]
we have $\widetilde{\capa}_1(A)=\beta$.
We show that $\beta=0$; suppose instead $\beta>0$. Then $\capa_1(A)>0$ by \eqref{eq:comparison of capacities},
and so by Theorem \ref{thm:fine Kellogg property}
there exists a point $x\in A\cap b_1 A$. Choose $i\in \Lambda$ such
that $x\in U_{i}$. Since $A\setminus U_i$ is $1$-thin at $x$ and $A$ is $1$-thick, we find
$0<r<\tfrac{1}{16} \diam X$ such that
\[
r\frac{\rcapa_1((A\setminus U_i)\cap B(x,r),B(x,2r))}{\mu(B(x,r))}<
\frac{1}{25C_S^2 C_d}\frac r2\frac{\rcapa_1(A\cap B(x,r/2),B(x,r))}{\mu(B(x,r/2))},
\]
and so
\begin{equation}\label{eq:thinness and thickness}
\rcapa_1((A\setminus U_i)\cap B(x,r),B(x,2r))<
\frac{1}{25C_S^2}\rcapa_1(A\cap B(x,r/2),B(x,r)).
\end{equation}
Then choose $k\in \N$ such that $B_k=B(y,s)$ with $d(y,x)<r/8$ and
$3r/4<s< 7r/8$. Now
\begin{align*}
&\rcapa_1((A\setminus U_i)\cap B_k,2B_k)\\
&\qquad \le 5C_S\rcapa_1((A\setminus U_i)\cap B_k,4B_k)\quad\textrm{by Lemma }\ref{lem:capacity wrt different balls}\\
&\qquad \le 5C_S\rcapa_1((A\setminus U_i)\cap B_k,B(x,2r))\quad\textrm{since } B(x,2r)\subset 4B_k\\
&\qquad \le 5C_S\rcapa_1((A\setminus U_i)\cap B(x,r),B(x,2r))\quad\textrm{since }  B_k\subset B(x,r)\\
&\qquad < \frac{1}{5C_S} \rcapa_1(A\cap B(x,r/2),B(x,r))\quad\textrm{by }\eqref{eq:thinness and thickness}\\
&\qquad \le \rcapa_1(A\cap B(x,r/2),B(x,2r))\\
&\qquad \le  \rcapa_1(A\cap B(x,r/2),2B_k)\\
&\qquad \le  \rcapa_1(A\cap B_k,2B_k).
\end{align*}
It follows that $\widetilde{\capa}_1(A\setminus U_i)<\widetilde{\capa}_1(A)=\beta$, a contradiction.
Hence $\beta=0$ and so
$\widetilde{\capa}_1\left(U\setminus \bigcup_{i\in I_{\infty}}U_i\right)=0$.
Then by \eqref{eq:comparison of capacities} we see that
$\capa_1\left(U\setminus \bigcup_{i\in I_{\infty}}U_i\right)=0$.
\end{proof}

\section{The Choquet property}\label{sec:choquet}

In this section we prove the fact that $1$-finely open sets
are $1$-quasiopen, and then we prove the Choquet property for the $1$-fine
topology. To achieve these,
we use the weak Cartan property proved in \cite{L-WC},
as well as the fine Kellogg property and the quasi-Lindel\"of principle
of the previous sections.

Recall that a set $U\subset X$ is $1$-quasiopen if for every $\eps>0$ there is an
open set $G\subset X$ such that $\capa_1(G)<\eps$ and $U\cup G$ is open.
Quasiopen sets have the following stability.

\begin{lemma}\label{lem:stability of quasiopen sets}
Let $U\subset X$ be a $1$-quasiopen set and let $A\subset X$ be $\mathcal H$-negligible.
Then $U\setminus A$ and $U\cup A$ are $1$-quasiopen sets.
\end{lemma}
\begin{proof}
Let $\eps>0$. Take an open set $G\subset X$ such that $\capa_1(G)<\eps$ and $U\cup G$ is an open set.
By \eqref{eq:null sets of Hausdorff measure and capacity} we know that $\capa_1(A)=0$, and since $\capa_1$
is an outer capacity, we find an open set $W\supset A$ such that $\capa_1(G)+\capa_1(W)<\eps$.
Now $(U\setminus A)\cup (G\cup W)=(U\cup G)\cup W$ is an open set with $\capa_1(G\cup W)<\eps$, so that
$U\setminus A$ is a $1$-quasiopen set. Similarly, $(U\cup A)\cup (G\cup W)=(U\cup G)\cup W$
is an open set, so that $U\cup A$ is  also a $1$-quasiopen set.
\end{proof}

The following fact about $1$-finely open and $1$-quasiopen sets is previously known.

\begin{proposition}[{\cite[Proposition 4.3]{L-Fed}}]\label{prop:quasiopen is finely open}
Every $1$-quasiopen set is the union of a $1$-finely open set and a $\mathcal H$-negligible set.
\end{proposition}

\begin{lemma}\label{lem:quasiopen contain qo and finely open set}
Let $V_0\subset X$ be $1$-quasiopen and let
$x\in V_0\setminus b_1(X\setminus V_0)$, that is, $X\setminus V_0$ is $1$-thin at
$x$. Then there exists a $1$-finely open and $1$-quasiopen set $V$ such that
$x\in V\subset V_0$.
\end{lemma}

\begin{proof}
By Proposition \ref{prop:quasiopen is finely open}, $V_0=W\cup N$ where
$\mathcal H(N)=0$
and $W$ is $1$-finely open, and $W$ is also $1$-quasiopen by
Lemma \ref{lem:stability of quasiopen sets}. Let $V:=W\cup \{x\}$.
If $\capa_1(\{x\})>0$, necessarily $x\in W$ and so $V=W$  is $1$-finely open and $1$-quasiopen.
Otherwise $V$ is also $1$-quasiopen
by Lemma \ref{lem:stability of quasiopen sets}. It is $1$-finely open since
\[
B(x,r)\setminus V\subset B(x,r)\setminus W\subset (B(x,r)\setminus V_0)\cup (B(x,r)\cap N)
\]
for all $r>0$, and so
\begin{align*}
&\limsup_{r\to 0}r\frac{\rcapa_1(B(x,r)\setminus V,B(x,2r))}{\mu(B(x,r))}\\
&\qquad\qquad \le\limsup_{r\to 0}
r\frac{\rcapa_1((B(x,r)\setminus V_0)\cup (B(x,r)\cap N),B(x,2r))}{\mu(B(x,r))}\\
&\qquad\qquad =\limsup_{r\to 0}
r\frac{\rcapa_1(B(x,r)\setminus V_0,B(x,2r))}{\mu(B(x,r))}\quad\textrm{by }
\eqref{eq:null sets of Hausdorff measure and capacity}\textrm{ and }
\eqref{eq:comparison of capacities}\\
&\qquad\qquad =0
\end{align*}
since $x\notin b_1(X\setminus V_0)$.
\end{proof}

Recall the definitions of the lower and upper approximate limits
$u^{\wedge}$ and $u^{\vee}$ from \eqref{eq:lower approximate limit}
and \eqref{eq:upper approximate limit}.
From \cite[Theorem 1.1]{LaSh} we get the following result,
which was proved earlier in the Euclidean setting in \cite[Theorem 2.5]{CDLP}.

\begin{proposition}\label{prop:quasisemicontinuity of BV}
Let $u\in\BV(X)$ and $\eps>0$. Then there exists an open set $G\subset X$ such that
$\capa_1(G)<\eps$ and
$u^{\wedge}|_{X\setminus G}$ is real-valued lower semicontinuous.
\end{proposition}

It is perhaps a curious fact that only now we need to talk about minimizers
for the first time; recall the definitions of $1$-minimizers and $1$-superminimizers from
Definition \ref{def:least gradient}.

\begin{theorem}[{\cite[Theorem 3.16]{L-WC}}]\label{thm:superminimizers are lsc}
Let $u$ be a $1$-superminimizer in an open set $\Om\subset X$.
Then $u^{\wedge}\colon\Om\to (-\infty,\infty]$ is lower semicontinuous.
\end{theorem}

We have the following weak Cartan property.

\begin{theorem}[{\cite[Theorem 5.3]{L-WC}}]\label{thm:weak Cartan property in text}
	Let $A\subset X$ and let $x\in X\setminus A$ be such that $A$
	is $1$-thin at $x$.
	Then there exist $R>0$ and $E_0,E_1\subset X$ such that $\ch_{E_0},\ch_{E_1}\in\BV(X)$,
	$\ch_{E_0}$ and $\ch_{E_1}$ are $1$-superminimizers in $B(x,R)$,
	$\max\{\ch_{E_0}^{\wedge},\ch_{E_1}^{\wedge}\}=1$ in $A\cap B(x,R)$,
	$\{\max\{\ch_{E_0}^{\vee},\ch_{E_1}^{\vee}\}>0\}$ is $1$-thin at $x$,
	and
	\[
	\lim_{r\to 0}r\frac{P(E_0,B(x,r))}{\mu(B(x,r))}=0,\qquad
	\lim_{r\to 0}r\frac{P(E_1,B(x,r))}{\mu(B(x,r))}=0.
	\]
\end{theorem}

Now we can prove a result that is an analog of the existence of \emph{p-strict subsets} for $p>1$,
see \cite[Lemma 3.3]{BBL-SS}. In fact, later we will only need the existence of the sets $V$
given in the proposition below, and not the functions $v$.

\begin{proposition}\label{prop:strict subsets}
	Let $U\subset X$ be $1$-finely open and let $x\in U$. Then there exists
	a $1$-finely open and $1$-quasiopen set $V$ such that $x\in V\subset U$,
	and a function $w\in\BV(X)$ such that $0\le w\le 1$ on $X$,
	$w^{\wedge}=1$ on $V$, and $\supp w\Subset U$.
\end{proposition}

\begin{proof}
The set $X\setminus U$ is $1$-thin at $x$.
Take $R>0$ and $E_0,E_1\subset X$ as given by Theorem \ref{thm:weak Cartan property in text}
with the choice $A=X\setminus U$.
Let $u:=\max\{\ch_{E_0}^{\wedge},\ch_{E_1}^{\wedge}\}\in\BV(X)$
(here we understand $u$ to be pointwise defined).
Now $0\le u\le 1$, $u$ is lower semicontinuous in $B(x,R)$ by
Theorem \ref{thm:superminimizers are lsc},
$u=1$ on $B(x,R)\setminus U$, and $u^{\vee}(x)=0$
since the set
$\{\max\{\ch_{E_0}^{\wedge},\ch_{E_1}^{\wedge}\}>0\}
\subset \{\max\{\ch_{E_0}^{\vee},\ch_{E_1}^{\vee}\}>0\}$ is $1$-thin at $x$,
and so it also has zero measure density at $x$, see \cite[Lemma 3.1]{L-Fed}.
Moreover,
$\{u^{\vee}>0\}
\subset \{\max\{\ch_{E_0}^{\vee},\ch_{E_1}^{\vee}\}>0\}$
and so
\begin{equation}\label{eq:thinness in Cartan property}
\lim_{r\to 0}r\frac{\rcapa_1(\{u^{\vee}>0\}\cap B(x,r),B(x,2r))}{\mu(B(x,r))}=0.
\end{equation}
Let $\eta\in \Lip_c(B(x,R))$ such that $0\le \eta\le 1$ on $X$ and $\eta=1$ on $B(x,R/2)$.
Let $v:=\eta(1-u)\in\BV(X)$ (still understood to be pointwise defined).
Then $0\le v\le 1$, $v$ is upper semicontinuous on $X$,
$v=0$ on $X\setminus U$, and $v^{\wedge}(x)=1$.
Moreover, by \eqref{eq:thinness in Cartan property},
\begin{equation}\label{eq:thinness in Cartan property consequence}
\lim_{r\to 0}r\frac{\rcapa_1(\{v^{\wedge}<1\}\cap B(x,r),B(x,2r))}{\mu(B(x,r))}=0.
\end{equation}
The set $V_0:=\{ v^{\wedge}>1/2\}$ contains $x$ and
is $1$-quasiopen by Proposition \ref{prop:quasisemicontinuity of BV}.
Moreover, $x\notin b_1(X\setminus V_0)$ by \eqref{eq:thinness in Cartan property consequence}.
By Lemma \ref{lem:quasiopen contain qo and finely open set} we find a $1$-finely
open and $1$-quasiopen set $V$ such that $x\in V\subset V_0$.
Let
\[
w(\cdot):=\min\left\{1,\left(4v(\cdot)-1\right)_+\right\}\in\BV(X)
\]
(now understood in the usual sense of a $\BV$ function, i.e. as a $\mu$-equivalence class).
Then $0\le w\le 1$ and $w^{\wedge}=1$ on $V$ (since it is a subset of $V_0$).
Moreover, $\supp w\subset \{v\ge 1/4\}$
by the upper semicontinuity of $v$. Thus $\supp w\Subset U$, and this
also guarantees that $V\subset U$.
\end{proof}

\begin{remark}
The above proof would be somewhat more straightforward if we knew
that the set $V_0$ itself was $1$-finely open; then we would not need
Lemma \ref{lem:quasiopen contain qo and finely open set}.
This would be the case if the
superminimizer functions $\ch_{E_0}^{\vee}$ and $\ch_{E_1}^{\vee}$ were
upper semicontinuous with respect to the $1$-fine topology,
but in general they are not, see \cite[Example 5.14]{L-WC}.
This is in contrast with the case $p>1$, where the $p$-fine topology
makes all $p$-superharmonic functions continuous. However, the fact that
$\{\max\{\ch_{E_0}^{\vee},\ch_{E_1}^{\vee}\}>0\}$ is $1$-thin at $x$ means that
$\ch_{E_0}^{\vee}$ and $\ch_{E_1}^{\vee}$ are $1$-finely continuous
at the point $x$, and this was enough for the proof to run through.
\end{remark}

\begin{theorem}\label{thm:finely open is quasiopen}
Every $1$-finely open set is $1$-quasiopen.
\end{theorem}
\begin{proof}
Let $U\subset X$ be $1$-finely open. For every $x\in U$, by Proposition \ref{prop:strict subsets}
we find a $1$-finely open and $1$-quasiopen set $V_x$ such that $x\in V_x\subset U$.
By the quasi-Lindel\"of principle (Theorem \ref{thm:quasi-Lindelof})
we find a countable subcollection $\{V_i\}_{i=1}^{\infty}$
and a $\capa_1$-negligible set $N\subset U$ such that $U=\bigcup_{i=1}^{\infty}V_i\cup N$.
The set $N$ is $1$-quasiopen since $\capa_1$ is an outer capacity, and
since a countable union of $1$-quasiopen sets is easily seen to be
$1$-quasiopen, $U$ is $1$-quasiopen.
\end{proof}

Combining this with Proposition \ref{prop:quasiopen is finely open}
and Lemma \ref{lem:stability of quasiopen sets}, we have a characterization
of $1$-quasiopen and $1$-finely open sets by means of each other.

\begin{corollary}
	A set $U\subset X$ is $1$-quasiopen if and only if it is the union of a $1$-finely
	open set and a $\mathcal H$-negligible set.
\end{corollary}

\begin{proposition}\label{prop:capacity of fine closure}
If $A\Subset D\subset X$, then
\[
\rcapa_1(A,D)=\rcapa_1(\overline{A}^1,D).
\]
\end{proposition}
\begin{proof}
By \cite[Proposition 3.3]{L-Fed}, we have
\[
\rcapa_1(A,D)=\rcapa_1(\overline{A}^1\cap D,D).
\]
Since clearly $\overline{A}^1\subset \overline{A}\subset D$, we have the result.
\end{proof}

Now we prove the Choquet property for $p=1$; for  the case $1<p<\infty$
see \cite[Theorem 1.2]{BBL-CCK}.

\begin{theorem}\label{thm:Choquet property}
Let $A\subset X$ and let $\eps>0$.
Then there exists an open set $W\subset X$ such that $W\cup b_1 A=X$
and $\capa_1(W\cap A)<\eps$.
\end{theorem}

\begin{proof}
Let $x\in X\setminus b_1 A$. Clearly for all $r>0$,
\[
b_1 A\cap B(x,r)\subset b_1(A\cap B(x,r))\subset \overline{A\cap B(x,r)}^1,
\]
and so by Proposition \ref{prop:capacity of fine closure},
\begin{align*}
&\limsup_{r\to 0} r\frac{\rcapa_1(b_1 A\cap B(x,r),B(x,2r))}{\mu(B(x,r))}\\
&\qquad\le \limsup_{r\to 0} r\frac{\rcapa_1(\overline{A\cap B(x,r)}^1,B(x,2r))}{\mu(B(x,r))}\\
&\qquad= \limsup_{r\to 0} r\frac{\rcapa_1(A\cap B(x,r),B(x,2r))}{\mu(B(x,r))}=0.
\end{align*}
Thus $X\setminus b_1 A$ is a $1$-finely open set, and then it is also $1$-quasiopen by Theorem \ref{thm:finely open is quasiopen}.
Take an open $G\subset X$ such that $\capa_1(G)<\eps$
and $(X\setminus b_1 A)\cup G=:W$ is open.
Now $W\cup b_1 A=X$, and by the fine Kellogg property
(Corollary \ref{cor:fine Kellogg}),
\[
\capa_1(W\cap A)\le \capa_1((A\setminus b_1 A)\cup G)=\capa_1(G)<\eps.
\]
\end{proof}

\begin{remark}
In the case $p>1$, one seems to need
a strong version of the Cartan property (involving only one superminimizer function, instead of two)
to deduce the Choquet property; see \cite{BBL-CCK,BBL-WC}. For $p=1$ we do not know whether such
a Cartan property holds,
though a result in that vein was given in the proof of \cite[Proposition 5.8]{L-Fed}.
However, the weak Cartan property is enough for proving
the existence of $1$-strict subsets as in Proposition \ref{prop:strict subsets},
and this combined with the quasi-Lindel\"of principle and the
fine Kellogg property is then enough for
proving the Choquet property.
However, the same would not work in the case $p>1$, since there the
Choquet property is needed for proving the fine Kellogg property.
\end{remark}

\paragraph{Acknowledgments.}
Part of the research for this paper was conducted during the author's visit to Link\"oping
University, and he wishes to thank this institution for its hospitality.
During the visit the research was funded by a grant from the Finnish
Cultural Foundation.

\noindent Address:\\

\noindent University of Jyvaskyla\\
Department of Mathematics and Statistics\\
P.O. Box 35, FI-40014 University of Jyvaskyla, Finland\\
E-mail: {\tt panu.k.lahti@jyu.fi}

\end{document}